\documentclass[12pt]{amsart}
\usepackage{amscd, amsfonts, amssymb, amsthm}
\usepackage{parskip, fullpage, verbatim}
\usepackage[all]{xypic}
\usepackage{ctable}
\usepackage{longtable}
\usepackage{array}
\usepackage{aliascnt}
\usepackage[colorlinks, breaklinks, linkcolor=blue]{hyperref} 
\usepackage{breakurl}
\usepackage{booktabs}
\usepackage{wasysym}

\newcommand\CA{{\mathcal A}} 
\newcommand\CB{{\mathcal B}}
 
\renewcommand\CD{{\mathcal D}}

\newcommand\CIF{{\mathcal {IF}}} 

\newcommand\CRF{{\mathcal {RF}}}

\newcommand\BBC{{\mathbb C}}

\newcommand\BBZ{{\mathbb Z}}

\newcommand {\GAP}{\textsf{GAP}}  
\newcommand {\CHEVIE}{\textsf{CHEVIE}}  
\newcommand {\Singular}{\textsf{SINGULAR}}  
\newcommand {\Sage}{\textsf{SAGE}}

\newcommand\GL{\operatorname{GL}}

\newcommand\Der{\operatorname{Der}}
\newcommand\pdeg{\operatorname{pdeg}}

\renewcommand\th{{^{\text{th}}}}

\numberwithin{equation}{section}

\theoremstyle{plain}
\newtheorem{lemma}[equation]{Lemma}
\newtheorem{theorem}[equation]{Theorem}

\newtheorem{corollary}[equation]{Corollary}
\newtheorem{proposition}[equation]{Proposition}
\theoremstyle{definition}
\newtheorem{definition}[equation]{Definition}
\newtheorem{remark}[equation]{Remark}
\newtheorem{example}[equation]{Example}

\begin{document}


\title[On Inductively free Restrictions of Reflection Arrangements]{On inductively free Restrictions of Reflection Arrangements}


\author[N. Amend]{Nils Amend}
\address
{Fakult\"at f\"ur Mathematik,
Ruhr-Universit\"at Bochum,
D-44780 Bochum, Germany}
\email{nils.amend@rub.de}

\author[T. Hoge]{Torsten Hoge}
\address
{Institut f\"ur Algebra, Zahlentheorie und Diskrete Mathematik,
Fakult\"at f\"ur Mathematik und Physik,
Leibniz Universit\"at Hannover,
Welfengarten 1,
30167 Hannover, Germany}
\email{hoge@math.uni-hannover.de}

\author[G. R\"ohrle]{Gerhard R\"ohrle}
\address
{Fakult\"at f\"ur Mathematik,
Ruhr-Universit\"at Bochum,
D-44780 Bochum, Germany}
\email{gerhard.roehrle@rub.de}


\begin{abstract}  
Let $W$ be a finite complex reflection group acting on the 
complex vector space $V$ and let $\CA(W) = (\CA(W), V)$ be 
the associated reflection arrangement. 
In \cite{hogeroehrle:inductivelyfree}, we classified 
all inductively free reflection
arrangements $\CA(W)$. The aim of this note is to 
extend this work by determining all inductively free 
restrictions of reflection arrangements.
\end{abstract}

\maketitle
\allowdisplaybreaks


\section{Introduction}

Let $W$ be a finite complex reflection group acting 
on the complex vector space $V$ and let $\CA = (\CA(W), V)$ be the associated
hyperplane arrangement of $W$. 
In \cite[Thm.~1.1]{hogeroehrle:inductivelyfree},
we classified all inductively free reflection arrangements, see 
Theorem \ref{thm:indfree1} below. 
See Definition \ref{def:indfree} below for the notion of 
an inductively free arrangement. 
Extending this earlier work, in
this note we classify all inductively free restrictions 
$\CA^X$, for $\CA$ a reflection arrangement and $X$ in
the intersection lattice $L(\CA)$ of $\CA$, see Theorem \ref{thm:indfree}. 
If $\CA^X$ is inductively free for every $X \in L(\CA)$, 
then $\CA$ is called hereditarily inductively free, see Definition
\ref{def:heredindfree}.

First we recall the main results 
from \cite[Thms.~1.1 and 1.2]{hogeroehrle:inductivelyfree}:

\begin{theorem}
\label{thm:indfree1}
For a finite complex reflection group  $W$,  
let $\CA = \CA(W)$ be its reflection arrangement.
Then the following hold:
\begin{itemize}
\item[(i)] $\CA$ is 
inductively free if and only if 
$W$ does not admit an irreducible factor
isomorphic to a monomial group 
$G(r,r,\ell)$ for $r, \ell \ge 3$, 
$G_{24}, G_{27}, G_{29}, G_{31}, G_{33}$, or $G_{34}$.
\item[(ii)] $\CA$ is inductively free
if and only if  $\CA$ is hereditarily inductively free.
\end{itemize}
\end{theorem}

In order to state our main results, we need 
a bit more notation: For fixed $r, \ell \geq 2$ 
and $0 \leq k \leq \ell$ we denote by
$\CA^k_\ell(r)$ the intermediate arrangements,
defined in \cite[\S 2]{orliksolomon:unitaryreflectiongroups}
(see also \cite[\S 6.4]{orlikterao:arrangements}), that 
interpolate between the reflection arrangements $\CA(G(r,r,\ell)) =
\CA^0_\ell(r)$ and $\CA(G(r,1,\ell)) = \CA^\ell_\ell(r)$, 
of the monomial groups $G(r,r,\ell)$ and $G(r,1,\ell)$, respectively. 
The arrangements $\CA^k_\ell(r)$ occur as 
restrictions of $\CA(G(r,r,\ell))$, 
\cite[Prop.\ 2.14]{orliksolomon:unitaryreflectiongroups}
(cf.\ \cite[Prop.~6.84]{orlikterao:arrangements}), 
see also Example \ref{ex:intermediate} below.
For $k \neq 0, \ell$, these are not reflection arrangements
themselves. See Section \ref{sec:akl} for further details.

Thanks to the compatibility of inductive freeness and 
products of arrangements, see Proposition \ref{prop:product-indfree}, 
as well as the product rule \eqref{eq:restrproduct} for 
restrictions in products, the question of inductive 
freeness of restrictions $\CA^X$ reduces readily to 
the case when $\CA$ is irreducible.
Thus we may assume that $W$ is irreducible.
In view of Theorem \ref{thm:indfree1} we can formulate our 
classification as follows:

\begin{theorem}
\label{thm:indfree}
Let $W$ be a finite, irreducible, complex 
reflection group with reflection arrangement 
$\CA = \CA(W)$ and let $X \in L(\CA)$. 
The restricted arrangement $\CA^X$ is inductively free 
if and only if one of the following holds:
\begin{itemize}
\item[(i)] 
$\CA$ is inductively free;
\item[(ii)] 
$W = G(r,r,\ell)$ and 
$\CA^X \cong \CA^k_p(r)$, where $p = \dim X$ and $p - 2 \leq k \leq p$; 
\item[(iii)] 
$W$ is one of $G_{24}, G_{27}, G_{29}, G_{31}, G_{33}$, or $G_{34}$ and $X \in L(\CA) \setminus \{V\}$ with  $\dim X \leq 3$.
\end{itemize}
\end{theorem}

Note that every 1- and 2-dimensional 
central arrangement is inductively free  
(Lemma \ref{lem:rank1and2}).
So we focus on higher-dimensional restrictions. 
It follows from Theorem \ref{thm:indfree1}(ii)
that for $\CA(W)$ inductively free, every 
restriction $\CA(W)^X$ is again inductively free.
Consequently, 
there are only two additional families of
inductively free restrictions, namely 
$\CA^{p-2}_p(r)$ and $\CA^{p-1}_p(r)$ for $p \ge 3$
(Theorem \ref{thm:indfree}(ii))
and thanks to 
Theorem \ref{thm:indfree1}(i) and 
the classification of the restrictions 
$\CA(W)^X$ for $W$ an exceptional reflection group from 
\cite[App.]{orliksolomon:unitaryreflectiongroups}
(cf.~\cite[App.\ C]{orlikterao:arrangements}), there are an additional 
8  inductively free, $3$-dimensional 
restrictions, up to isomorphism, 
(Theorem \ref{thm:indfree}(iii)),
see \S \ref{subsec:pfthmindfree}.

Our next result asserts that the equivalence of Theorem \ref{thm:indfree1}(ii)
extends to restrictions.

\begin{theorem}
\label{thm:heredfree}
For a finite complex reflection group $W$,  
let $\CA = \CA(W)$ be its reflection arrangement
and let $X \in L(\CA)$.
Then $\CA^X$ is inductively free
if and only if  $\CA^X$ is hereditarily inductively free.
\end{theorem}

The paper is organized as follows: In \S \ref{ssect:recoll} 
we recall the required notation and some facts about inductively free
arrangements from \cite{orlikterao:arrangements} and 
\cite{hogeroehrle:inductivelyfree}. 
In Section \ref{sec:akl} we study the intermediate
arrangements $\CA_\ell^k(r)$, see Theorem \ref{thm:akl}, 
and in Section \ref{sec:proofs} we
prove Theorems \ref{thm:indfree} and \ref{thm:heredfree}.
We close with a result on recursively free restrictions
of reflection arrangements, 
Corollary \ref{cor:recfree}.

For general information about arrangements and reflection groups 
we refer the reader to \cite{orliksolomon:unitaryreflectiongroups},
\cite{orlikterao:arrangements} and
\cite{bourbaki:groupes}. In this article we use the classification 
and labeling of the irreducible unitary reflection groups
due to Shephard and Todd, \cite{shephardtodd}.

\section{Recollections}
\label{ssect:recoll}

\subsection{Hyperplane Arrangements}
\label{ssect:hyper}

Suppose $V$ is a finite dimensional complex vector space. 
By a hyperplane arrangement in $V$ we mean a finite set $\CA$ of
hyperplanes in $V$. Such an arrangement is denoted $(\CA,V)$ 
or simply $\CA$, when there is no ambiguity. If $\dim V = \ell$
we call $\CA$ an $\ell$-arrangement and the empty $\ell$-arrangement 
is denoted by $\Phi_\ell$.

If $X$ is a subspace of $V$, there are two natural arrangements 
associated to $X$,
\begin{itemize}
 \item the subarrangement $\CA_X = (\CA_X,V)$ of $\CA$ defined 
by $\CA_X := \{H \in \CA \;|\; X \subseteq H\}$,
 \item the restriction  $\CA^X = (\CA^X,X)$ of $\CA$ to  $X$ 
defined by $\CA^X := \{H \cap X \;|\; H \in \CA \setminus \CA_X\}$.
\end{itemize}

We only consider \emph{central} arrangements, 
i.e.\ $0 \in  \bigcap_{H \in \CA}H$.
The \emph{lattice} $L(\CA)$  of the arrangement $\CA$ consists of all 
intersections of hyperplanes in $\CA$.
Note that $L(\CA)$ contains $V$ as the empty intersection.

Let $H_0 \in \CA$ (for $\CA \neq \Phi_\ell$) and
define $\CA' := \CA \setminus\{ H_0\}$,
and $\CA'' := \CA^{H_0}$.
Then $(\CA, \CA', \CA'')$ is a \emph{triple} of arrangements (with distinguished hyperplane $H_0$),
\cite[Def.~1.14]{orlikterao:arrangements}.

The \emph{product}
$\CA = (\CA_1 \times \CA_2, V_1 \oplus V_2)$ 
of two arrangements $(\CA_1, V_1), (\CA_2, V_2)$
is defined by
\begin{equation}
\label{eq:product}
\CA := \CA_1 \times \CA_2 = \{H_1 \oplus V_2 \mid H_1 \in \CA_1\} \cup 
\{V_1 \oplus H_2 \mid H_2 \in \CA_2\},
\end{equation}
see \cite[Def.~2.13]{orlikterao:arrangements}.

An arrangement $\CA$ is called \emph{reducible},
if it is of the form $\CA = \CA_1 \times \CA_2$, where 
$\CA_i \ne \Phi_0$ for $i=1,2$, else $\CA$
is 
\emph{irreducible}, 
\cite[Def.~2.15]{orlikterao:arrangements}.

If $\CA = \CA_1 \times \CA_2$ is a product, 
then by \cite[Prop.~2.14]{orlikterao:arrangements}
there is a lattice isomorphism
\[
 L(\CA_1) \times L(\CA_2) \cong L(\CA) \quad \text{by} \quad
(X_1, X_2) \mapsto X_1 \oplus X_2.
\]
With \eqref{eq:product}, it is easy to see that
for $X =  X_1 \oplus X_2 \in L(\CA)$, we have 
$\CA _X =  ({\CA_1})_{X_1} {\times} ({\CA_2})_{X_2}$
and 
\begin{equation}
\label{eq:restrproduct}
\CA^X = \CA_1^{X_1} \times \CA_2^{X_2}.
\end{equation}

\subsection{Free Arrangements}
\label{ssect:free}

Let $S = S(V^*)$ be the symmetric algebra of the dual space $V^*$ of $V$.
If $x_1, \ldots , x_\ell$ is a basis of $V^*$, then we identify $S$ with 
the polynomial ring $\BBC[x_1, \ldots , x_\ell]$.
By denoting the $\BBC$-subspace of $S$
consisting of the homogeneous polynomials of degree $p$ (and $0$) by $S_p$,
we see that there is a natural $\BBZ$-grading
$S = \oplus_{p \in \BBZ}S_p$, where
$S_p = 0$ for $p < 0$.

Let $\Der(S)$ be the $S$-module of $\BBC$-derivations of $S$
and for $i = 1, \ldots, \ell$
define $D_i := \partial/\partial x_i$.
Now $D_1, \ldots, D_\ell$ is an $S$-basis of $\Der(S)$ and
we call $\theta \in \Der(S)$
\emph{homogeneous of polynomial degree $p$}
provided 
$\theta = \sum_{i=1}^\ell f_i D_i$, 
where $f_i \in S_p$ for each $1 \leq i \leq \ell$.
In this case we write $\pdeg \theta = p$.
By defining $\Der(S)_p$ to be the $\BBC$-subspace of $\Der(S)$ consisting 
of all homogeneous derivations of polynomial degree $p$,
we see that $\Der(S)$ is a graded $S$-module:
$\Der(S) = \oplus_{p\in \BBZ} \Der(S)_p$.

Following \cite[Def.~4.4]{orlikterao:arrangements}, 
we define the $S$-submodule $D(f)$ of $\Der(S)$ for $f \in S$ by
\[
D(f) := \{\theta \in \Der(S) \mid \theta(f) \in f S\} .
\]

If $\CA$ is an arrangement in $V$,
then for every $H \in \CA$ we may fix $\alpha_H \in V^*$ with
$H = \ker(\alpha_H)$.
We call $Q(\CA) := \prod_{H \in \CA} \alpha_H \in S$
the \emph{defining polynomial} of $\CA$.

The \emph{module of $\CA$-derivations} is the $S$-submodule of $\Der(S)$ 
defined by 
\[
D(\CA) := D(Q(\CA)).
\]
The arrangement $\CA$ is said to be \emph{free} if the module of $\CA$-derivations
$D(\CA)$ is a free $S$-module.

Note that $D(\CA)$
is a graded $S$-module $D(\CA) = \oplus_{p\in \BBZ} D(\CA)_p$,
where $D(\CA)_p = D(\CA) \cap \Der(S)_p$, see
\cite[Prop.~4.10]{orlikterao:arrangements}.
If $\CA$ is a free $\ell$-arrangement, 
then by \cite[Prop.~4.18]{orlikterao:arrangements}
the $S$-module $D(\CA)$ admits a basis of $\ell$ homogeneous derivations $\theta_1, \ldots, \theta_\ell$.
While these derivations are not unique, their polynomial 
degrees $\pdeg \theta_i$ are unique (up to ordering).
The \emph{exponents} of the free arrangement $\CA$ is the multiset
$\exp\CA := \{\pdeg \theta_1, \ldots, \pdeg \theta_\ell\}$.

The fundamental \emph{Addition-Deletion Theorem} 
due to Terao  \cite{terao:freeI} plays a 
crucial role in the study of free arrangements, 
\cite[Thm.\ 4.51]{orlikterao:arrangements}.

\begin{theorem}
\label{thm:add-del}
Suppose $\CA \neq \Phi_\ell$ and
let $(\CA, \CA', \CA'')$ be a triple of arrangements. Then any 
two of the following statements imply the third:
\begin{itemize}
\item $\CA$ is free with $\exp\CA = \{ b_1, \ldots , b_{\ell -1}, b_\ell\}$;
\item $\CA'$ is free with $\exp\CA' = \{ b_1, \ldots , b_{\ell -1}, b_\ell-1\}$;
\item $\CA''$ is free with $\exp\CA'' = \{ b_1, \ldots , b_{\ell -1}\}$.
\end{itemize}
\end{theorem}

Suppose that  $\CA \ne \Phi_\ell$ and consider a triple $(\CA,\CA', \CA'')$.
We recall the following very useful criterion from 
\cite[Cor.~2.7]{hogeroehrle:inductivelyfree}:

\begin{lemma}
\label{lem:q}
Suppose that $\CA$ and $\CA''$ are free
and $\exp \CA'' \not\subseteq \exp \CA$.
Then $\CA'$ is not free.
\end{lemma}

\subsection{Inductively and Recursively Free Arrangements}
\label{ssect:indfree}

Theorem \ref{thm:add-del} motivates the notion of 
\emph{inductively free} arrangements, 
cf.\ \cite[Def.~4.53]{orlikterao:arrangements}:

\begin{definition}
\label{def:indfree}
The class $\CIF$ of \emph{inductively free} arrangements 
is the smallest class of arrangements subject to
\begin{itemize}
\item[(i)] $\Phi_\ell \in \CIF$ for each $\ell \ge 0$;
\item[(ii)] if there exists a hyperplane $H_0 \in \CA$ such that both
$\CA'$ and $\CA''$ belong to $\CIF$, and $\exp \CA '' \subseteq \exp \CA'$, 
then $\CA$ also belongs to $\CIF$.
\end{itemize}
\end{definition}

\begin{remark}
\label{rem:indtable}
An inductively free arrangement $\CA$ can be described by means of a so called 
\emph{induction table}, cf.~\cite[\S 4.3, p.~119]{orlikterao:arrangements}.
In this process we successively add hyperplanes to an inductively free arrangement
$\CA_0$, ensuring that in each step part (ii) of Definition \ref{def:indfree} is satisfied.
This process is referred to as \emph{induction of hyperplanes}.
It amounts to choosing a total order on $\CA\setminus\CA_0$, say 
$\CA = \CA_0 \cup \{H_1, \ldots, H_n\}$, 
so that the subarrangements 
$\CA_0$, $\CA_i := \CA_0 \cup \{H_1, \ldots, H_i\}$
and the restrictions $\CA_i^{H_i}$ are inductively free
for all $1 \leq i \leq n$.
In the associated induction table we record in the $i\th$ row the information 
of the $i\th$ step of this process, by 
listing $\exp \CA_i' = \exp\CA_{i-1}$, 
the defining form $\alpha_{H_i}$ of $H_i$, 
as well as $\exp\CA_i'' = \exp\CA_i^{H_i}$, 
for $i = 1, \ldots, n$.
E.g.\ see Tables \ref{table3} -- \ref{table8} below.
\end{remark}

Next we recall the compatibility 
of products and inductive freeness from
\cite[Prop.~2.10]{hogeroehrle:inductivelyfree}:

\begin{proposition}
\label{prop:product-indfree}
Let $\CA_1, \CA_2$ be two arrangements.
Then  $\CA = \CA_1 \times \CA_2$ is 
inductively free if and only if both 
$\CA_1$ and $\CA_2$ are 
inductively free and in that case
$\exp\CA = \{\exp\CA_1, \exp\CA_2\}$.
\end{proposition}

There is an even stronger notion of freeness,
cf.~\cite[\S 6.4]{orlikterao:arrangements}.

\begin{definition}
\label{def:heredindfree}
The arrangement $\CA$ is called 
\emph{hereditarily inductively free} provided 
that $\CA^X$ is inductively free for each $X \in L(\CA)$.
\end{definition}

Note that if $\CA$ is hereditarily inductively free, 
it is inductively free as $V \in L(\CA)$ and $\CA^V = \CA$.

The compatibility with 
products from Proposition \ref{prop:product-indfree}
also extends to this stronger notion, cf.\ 
\cite[Cor.~2.12]{hogeroehrle:inductivelyfree}.

It is easy to see that any $1$- or $2$-arrangement is
hereditarily inductively free, 
\cite[Ex.~2.13, Lem.~2.14]{hogeroehrle:inductivelyfree}:

\begin{lemma}
\label{lem:rank1and2}
Any central 1- or 2-arrangement is hereditarily inductively free.
\end{lemma}

These two stronger notions of freeness still coincide for $3$-arrangements, 
\cite[Lem.~2.15]{hogeroehrle:inductivelyfree}:

\begin{lemma}
\label{lem:hif}
Let $\CA$ be a $3$-arrangement. Then $\CA$ is inductively free if and only 
if $\CA$ is hereditarily inductively free.
\end{lemma}

There is another notion of freeness motivated by Theorem \ref{thm:add-del}, cf.\ 
\cite[Def.~4.60]{orlikterao:arrangements}:

\begin{definition}
\label{def:recfree}
The class $\CRF$ of \emph{recursively free} arrangements 
is the smallest class of arrangements subject to
\begin{itemize}
\item[(i)] $\Phi_\ell \in \CRF$ for each $\ell \geq 0$;
\item[(ii)] if there exists a hyperplane $H_0 \in \CA$ such that both
$\CA'$ and $\CA''$ belong to $\CRF$, and $\exp \CA'' \subseteq \exp \CA'$, 
then $\CA$ also belongs to $\CRF$;
\item[(iii)] if $\CA \in \CRF$ and there exists a hyperplane $H_0 \in \CA$
such that $\CA'' \in \CRF$ and $\exp \CA'' \subseteq \exp \CA$, then $\CA'$
also belongs to $\CRF$.
\end{itemize}
\end{definition}

\subsection{Reflection Arrangements}
\label{ssect:refl}

Let $W \subseteq \GL(V)$ be a finite, 
complex reflection group acting on the complex vector space $V=\BBC^\ell$.
The \emph{reflection arrangement} of $W$ in $V$ is the 
hyperplane arrangement $\CA = \CA(W)$ 
consisting of the reflecting hyperplanes 
of the elements in $W$ acting as reflections on $V$.

Terao \cite{terao:freeI} has shown that every 
reflection arrangement $\CA = \CA(W)$ is free 
and that the exponents of $\CA$
coincide with the coexponents of $W$, see also 
\cite[Prop.~6.59 and Thm.~6.60]{orlikterao:arrangements}. 

Note that the reflection arrangements 
of $G(r,1,\ell)$ and $G(r,p,\ell)$ 
with $r, \ell \geq 2$ and $p \neq r$ are identical, cf.\ 
\cite[\S 6.4]{orlikterao:arrangements}.

\section{The intermediate arrangements $\CA_\ell^k(r)$}
\label{sec:akl}

Orlik and Solomon defined intermediate 
arrangements $\CA^k_\ell(r)$ in 
\cite[\S 2]{orliksolomon:unitaryreflectiongroups}
(cf.\ \cite[\S 6.4]{orlikterao:arrangements}) which
interpolate between the
reflection arrangements of $G(r,r,\ell)$ and $G(r,1,\ell)$. 
These play a pivotal role in our analysis, since they 
show up as restrictions of the reflection arrangement
of $G(r,r,\ell)$, 
\cite[Prop.\ 2.14]{orliksolomon:unitaryreflectiongroups} 
(cf.~\cite[Prop.\ 6.84]{orlikterao:arrangements}),
see also Example \ref{ex:intermediate}.

For 
$\ell \geq 2$ and $0 \leq k \leq \ell$ the defining polynomial of
$\CA^k_\ell(r)$ is given by
$$Q(\CA^k_\ell(r)) = x_1 \cdots x_k\prod\limits_{\substack{1 \leq i < j \leq \ell\\ 0 \leq n < r}}(x_i - \zeta^nx_j),$$
where $\zeta$ is a primitive $r\th$ root of unity,
so that 
$\CA^\ell_\ell(r) = \CA(G(r,1,\ell))$ and 
$\CA^0_\ell(r) = \CA(G(r,r,\ell))$. Next we recall
\cite[Props.\ 2.11,  2.13]{orliksolomon:unitaryreflectiongroups}
(cf.~\cite[Props.~6.82,  6.85]{orlikterao:arrangements}):

\begin{proposition}
\label{prop:intermediate}
Let $\CA = \CA^k_\ell(r)$.
\begin{enumerate}
 \item[(i)] $\CA$ is free with $\exp\CA = \{1, r + 1, \ldots, (\ell - 2)r + 1, (\ell - 1)r - \ell + k + 1\}$.
 \item[(ii)] Let $H \in \CA$. The type of $\CA^H$ is given in Table \ref{table2}.
\end{enumerate}
\end{proposition}

\begin{table}[ht!b]
\renewcommand{\arraystretch}{1.5}
\begin{tabular}{llll}\hline
 $k$ & \multicolumn{2}{l}{$\alpha_H$} & Type of $\CA^H$\\ \hline
 $0$ & arbitrary & & $\CA^1_{\ell - 1}(r)$\\
 $1, \ldots, \ell - 1$ & $x_i - \zeta x_j$ & $1 \leq i < j \leq k < \ell$ & $\CA^{k - 1}_{\ell - 1}(r)$\\
 $1, \ldots, \ell - 1$ & $x_i - \zeta x_j$ & $1 \leq i \leq k < j \leq \ell$ & $\CA^k_{\ell - 1}(r)$\\
 $1, \ldots, \ell - 1$ & $x_i - \zeta x_j$ & $1 \leq k < i < j \leq \ell$ & $\CA^{k + 1}_{\ell - 1}(r)$\\
 $1, \ldots, \ell - 1$ & $x_i$ & $1 \leq i \leq \ell$ & $\CA^{\ell - 1}_{\ell - 1}(r)$\\
 $\ell$ & arbitrary & & $\CA^{\ell - 1}_{\ell - 1}(r)$\\ \hline
\end{tabular}
\bigskip
\caption{Restriction types of $\CA^k_\ell(r)$}
\label{table2}
\end{table}

The following example shows that every  intermediate 
arrangement does occur as a restriction of
the reflection arrangement of $W = G(r,r,\ell)$
for a suitable $\ell$.

\begin{example}
\label{ex:intermediate}
Let $1 \leq n \leq p$, $\ell = p + n$ and $r \geq 3$ and let $W = G(r,r,\ell)$. For an $r\th$ root of unity $\zeta$ and for $1
\leq i < j \leq \ell$ let $H_{i,j}(\zeta) = \ker(x_i - \zeta x_j)$ be a hyperplane in $\CA = \CA^0_\ell(r) = \CA(W)$. 
Define $X := \bigcap_{i = 1}^{n}H_{2i-1, 2i}(\zeta) \in L(\CA)$.
Now $\dim X = \ell - n = p$, 
and by \cite[Prop.\ 2.14]{orliksolomon:unitaryreflectiongroups}
(cf.~\cite[Prop.~6.84]{orlikterao:arrangements}), 
we have $\CA^X \cong \CA^n_p(r)$.
\end{example}

\begin{lemma}
\label{lem:indfree1}
$\CA^{\ell - 2}_\ell(r)$ is inductively free.
\end{lemma}

\begin{proof}
We argue by induction on $\ell$.
As the result is clear for $\ell = 2$, 
by Lemma \ref{lem:rank1and2},
we may assume that $\ell \geq 3$ and 
that $\CA^{\ell - 3}_{\ell - 1}(r)$ is inductively free.
The subarrangement 
$\CA^{\ell - 3}_{\ell - 1}(r) \times \Phi_1$ of $\CA^{\ell - 2}_\ell(r)$ 
is inductively free with exponents
$\{\exp\CA^{\ell - 3}_{\ell - 1}(r), 0\}$, 
by Proposition \ref{prop:product-indfree}.
Now we use induction of hyperplanes to show that 
$\CA^{\ell - 2}_\ell(r)$ is inductively free, see Remark \ref{rem:indtable}.

The defining polynomial of $\CA^{\ell - 3}_{\ell - 1}(r)$ is given by
  $$Q^{\ell - 3}_{\ell - 1} := x_1x_2 \cdots x_{\ell-3} \prod\limits_{1 \leq i < j \leq \ell - 1}(x_i^r - x_j^r) = x_1x_2
  \cdots x_{\ell-3} \prod\limits_{1 \leq i < j \leq \ell - 1}\left(\prod\limits_{m = 0}^{r - 1}(x_i - \zeta^mx_j)\right).$$
We now add to the inductively free subarrangement $\CA^{\ell - 3}_{\ell - 1}(r) \times \Phi_1$ of $\CA^{\ell - 2}_\ell(r)$
the hyperplanes $\ker(x_{\ell - 2})$ and $\ker(x_i - \zeta^mx_\ell)$ for $1 \leq i \leq \ell - 1$ and $0 \leq m < r$
successively.

The additional factors (other than the ones in $Q^{\ell - 3}_{\ell - 1}$) of the defining polynomial $Q^{\ell - 2}_\ell$ of
$\CA^{\ell - 2}_\ell(r)$ are $P := \{x_{\ell - 2}, x_i - \zeta^mx_\ell\;|\;1 \leq i \leq \ell - 1, 0 \leq m < r\}$.
Define $\CA_{-1} := \CA^{\ell - 3}_{\ell - 1}(r) \times \Phi_1$ and $\CA_i := \CA_{i - 1} \cup \{H_i\}$ for $0 \leq i \leq (\ell
- 1)r$, where $H_0 := \ker(x_{\ell - 2})$ and $H_{(k - 1)r + j + 1} := \ker(x_{k} - \zeta^jx_\ell)$ for $1 \leq k \leq \ell - 1$
and $0 \leq j < r$. Thus we have $\CA_{(\ell - 1)r} = \CA^{\ell - 2}_\ell(r)$.
In Table \ref{table1},
we display this induction of hyperplanes, i.e.\ we record in the $i\th$ row
(starting with row number $0$) the information of the $i\th$ step of the induction process, by listing $\exp\CA_i' =
\exp\CA_{i - 1}$, the defining form $\alpha_{H_i}$ of $H_i$, as well as $\exp\CA_i'' = \exp\CA_i^{H_i}$. Observe that
$\exp\CA_0' = \exp\CA_{-1} = \{\exp(\CA^{\ell - 2}_{\ell - 2}(r) \times \Phi_1), (\ell - 2)r - 1\}$ in the first row of the table, using
Proposition \ref{prop:intermediate}(i).

Thus, the restriction of $\CA_0$ to $\ker(x_{\ell - 2})$ results in the substitution $x_{\ell - 2} = 0$ and restricting an
intermediate subarrangement to $\ker(x_i - \zeta^mx_\ell)$ results in the substitution $x_\ell = \zeta^{-m}x_i$.
So we get 
$\{x_1, \ldots, x_{\ell - 3}, x_{\ell - 1}, x_j - \zeta^mx_k \;|\; 1 \leq j < k \leq \ell - 1, k \neq \ell - 2, 0 \leq m < r - 1\}$ 
as defining terms for $\CA_0^{H_0}$, so that  $\CA_0'' \cong \CA^{\ell - 2}_{\ell - 2}(r) \times \Phi_1$.
The defining terms for $\CA_i''$ with $1 \leq i \leq (\ell - 2)r + 1$ are $\{x_1, \ldots, x_{\ell - 2}, x_j - \zeta^mx_k \;|\;
1 \leq j < k \leq \ell - 1, 0 \leq m < r - 1\}$, hence $\CA_i'' \cong \CA^{\ell - 2}_{\ell - 1}(r)$ for all $1 \leq i \leq (\ell
- 2)r + 1$.
For $(\ell - 2)r + 2 \leq i \leq (\ell - 1)r$, we get $\{x_1, \ldots, x_{\ell - 1}, x_j - \zeta^mx_k \;|\; 1 \leq j < k \leq \ell
- 1, 0 \leq m < r - 1\}$ as defining terms for $\CA_i''$, so $\CA_i'' \cong \CA^{\ell - 1}_{\ell - 1}(r)$ for all $(\ell - 2)r +
2 \leq i \leq (\ell - 1)r$.

\begin{table}[ht!b]
\renewcommand{\arraystretch}{1.5}
\begin{tabular}{lll}\hline
  $\exp\CA_i'$ & $\alpha_{H_i}$ & $\exp\CA_i''$\\ \hline\hline
  $\exp(\CA^{\ell - 2}_{\ell - 2}(r) \times \Phi_1), (\ell - 2)r - 1$ & $x_{\ell - 2}$ & $\exp(\CA^{\ell - 2}_{\ell - 2}(r) \times \Phi_1)$ \\
  $\exp\CA^{\ell - 2}_{\ell - 1}(r), 0$ & $x_1 - x_\ell$ & $\exp\CA^{\ell - 2}_{\ell - 1}(r)$ \\
  $\vdots$ & $\vdots$ & $\vdots$ \\
  $\exp\CA^{\ell - 2}_{\ell - 1}(r), r - 1$ & $x_1 - \zeta^{r - 1}x_\ell$ & $\exp\CA^{\ell - 2}_{\ell - 1}(r)$ \\
  $\exp\CA^{\ell - 2}_{\ell - 1}(r), r$ & $x_2 - x_\ell$ & $\exp\CA^{\ell - 2}_{\ell - 1}(r)$ \\
  $\vdots$ & $\vdots$ & $\vdots$ \\
  $\exp\CA^{\ell - 2}_{\ell - 1}(r), 2r - 1$ & $x_2 - \zeta^{r - 1}x_\ell$ & $\exp\CA^{\ell - 2}_{\ell - 1}(r)$ \\
  $\vdots$ & $\vdots$ & $\vdots$ \\
  $\exp\CA^{\ell - 2}_{\ell - 1}(r), (\ell - 2)r$ & $x_{\ell - 1} - x_\ell$ & $\exp\CA^{\ell - 2}_{\ell - 1}(r)$ \\
  $\exp\CA^{\ell - 1}_{\ell - 1}(r), (\ell - 2)r$ & $x_{\ell - 1} - \zeta x_\ell$ & $\exp\CA^{\ell - 1}_{\ell - 1}(r)$ \\
  $\vdots$ & $\vdots$ & $\vdots$ \\
  $\exp\CA^{\ell - 1}_{\ell - 1}(r), (\ell - 1)r - 2$ & $x_{\ell - 1} - \zeta^{r - 1}x_\ell$ & $\exp\CA^{\ell - 1}_{\ell - 1}(r)$ \\
  $\exp\CA^{\ell - 1}_{\ell - 1}(r), (\ell - 1)r - 1$ && \\ \hline
\end{tabular}
\smallskip
\caption{Induction Table for $\CA^{\ell - 2}_\ell(r)$}
\label{table1}
\end{table}

The exponents in Table \ref{table1} can be determined using 
Theorem \ref{thm:add-del} and Proposition \ref{prop:intermediate}.
\end{proof}

\begin{corollary}
\label{indfree2}
$\CA^{\ell - 1}_\ell(r)$ is inductively free.
\end{corollary}

\begin{proof}
The arrangement $\CA = \CA^{\ell - 1}_{\ell}(r)$ can be obtained from $\CA^{\ell - 2}_{\ell}(r)$ by adding the hyperplane
$\ker(x_{\ell - 1})$. Regarding this hyperplane we get $\CA' = \CA^{\ell - 2}_{\ell}(r)$ and
$\CA'' \cong \CA^{\ell - 1}_{\ell - 1}(r)$, by Proposition \ref{prop:intermediate}(ii). 
Also $\exp\CA^{\ell - 2}_{\ell}(r) = \{\exp\CA^{\ell - 1}_{\ell - 1}(r), (\ell - 1)r - 1\}$, 
by Proposition \ref{prop:intermediate}(i).
Thus $\exp \CA'' \subseteq \exp \CA'$.
Moreover, it follows from Lemma \ref{lem:indfree1} and 
Theorem \ref{thm:indfree1} that $\CA'$ and $\CA''$ are
inductively free, respectively.
Thus $\CA$ is inductively free.
\end{proof}

\begin{lemma}
\label{lem:notindfree}
$\CA^k_\ell(r)$ is not inductively free for $0 \le k \le \ell - 3$ and $r \geq 3$.
\end{lemma}

\begin{proof}
Thanks to Theorem \ref{thm:indfree1}(i), 
$\CA^0_\ell(r)$ is not inductively free. We now argue by induction
on $k$.
Suppose $1 \leq k \leq \ell - 3$ and $\CA = \CA^k_\ell(r)$.
Let $H \in \CA$.
If  $H = \ker(x_i)$ for some $1 \leq i \leq k$, then
$\CA' \cong \CA^{k - 1}_\ell(r)$ which 
 is not inductively free, by induction hypothesis.
So let $H \neq \ker(x_i)$ for $1 \leq i \leq k$ and let $(\CA, \CA', \CA'')$ be the triple of arrangements
corresponding to $H$.
Using Proposition \ref{prop:intermediate}, 
we see that in this case $\exp \CA'' \not\subseteq \exp \CA$,
since $k < \ell - 2$ and $r > 2$.
Thus, by Lemma \ref{lem:q}, $\CA'$ is not (inductively) free.

Thus for any choice of $H$ in $\CA$, 
the subarrangement $\CA'$ is not inductively free. 
Hence $\CA = \CA^k_\ell(r)$ is not inductively free.
\end{proof}

Our key result in this section classifies all 
inductively free arrangements among the $\CA^k_\ell(r)$.

\begin{theorem}
\label{thm:akl}  
Suppose $r \geq 2$, $\ell \geq 3$ and $0 \leq k \leq \ell$.
\begin{itemize} 
\item[(i)] 
$\CA^k_\ell(r)$ is inductively free if and only if $r = 2$ or $r \geq 3$ and $\ell - 2 \leq k \leq \ell$.
\item[(ii)] 
$\CA^k_\ell(r)$ is recursively free.
\end{itemize}
\end{theorem}

\begin{proof}
(i). 
It follows from \cite[Ex.~2.6]{jambuterao:supersolvable} that 
$\CD^k_\ell = \CA^k_\ell(2)$ is inductively free for
each $0 \leq k \leq \ell$. Now let $r \geq 3$.
By Theorem \ref{thm:indfree1}(i), Lemma \ref{lem:indfree1} and Corollary \ref{indfree2}, 
$\CA^{\ell}_\ell(r)$,  $\CA^{\ell - 2}_\ell(r)$ and $\CA^{\ell - 1}_\ell(r)$ are inductively
free.
By Lemma \ref{lem:notindfree}, $\CA^k_\ell(r)$ is not inductively free for $0 \leq k \leq \ell - 3$.

(ii).
Any $2$-arrangement is inductively free (and thus recursively free),
by Lemma \ref{lem:rank1and2}.
By part (i), the arrangements $\CA^k_\ell(r)$ 
are inductively free (and thus recursively free) for
$\ell - 2 \leq k \leq \ell$ with $\ell \geq 3$ and arbitrary $r$.

Now we use reverse induction on $k$ starting at $k = \ell - 2$.
Suppose $1 \leq k \leq \ell - 2$ and 
let $\CA := \CA^k_\ell(r)$ 
(which is recursively free by induction hypothesis) and
$H := \ker(x_k) \in \CA$.
Using Proposition \ref{prop:intermediate}, 
we see that $\CA'' = \CA^H \cong \CA^{\ell - 1}_{\ell - 1}(r)$ and
$\exp \CA'' \subseteq \exp\CA$.
Hence $\CA' = \CA\backslash\{H\} = \CA^{k - 1}_\ell(r)$ is recursively free.
This proves (ii).
\end{proof}

\section{Proofs of Theorems \ref{thm:indfree}
and \ref{thm:heredfree}}
\label{sec:proofs}

By Steinberg's Theorem \cite{steinberg}
(cf.\ \cite[Thm.\ 6.25]{orlikterao:arrangements}), 
the pointwise stabilizer $W_X$ of $X$ in $L(\CA)$ 
is again a complex reflection group.
So following 
\cite{orliksolomon:unitaryreflectiongroups} and 
\cite[\S 6.4, App.~C]{orlikterao:arrangements}, 
we label the $W$-orbit of $X \in L(\CA)$ by the type $T$ say,
of $W_X$. Therefore, we denote such a restriction $\CA(W)^X$ 
by the pair $(W, T)$ whenever convenient.

\subsection{Proof of Theorem \ref{thm:indfree}}
\label{subsec:pfthmindfree}
Part (i) of Theorem \ref{thm:indfree} is 
simply Theorem \ref{thm:indfree1}(ii).
Part (ii)  follows from 
Theorem \ref{thm:akl}(i) and 
\cite[Prop.\ 2.14]{orliksolomon:unitaryreflectiongroups}
(cf.~\cite[Prop.~6.84]{orlikterao:arrangements}).

Finally, we consider part (iii), i.e., the 
irreducible reflection groups of exceptional type
with non-inductively free reflection arrangement.

Thanks to  \cite{hogeroehrle:free}, 
the restrictions $(G_{34}, A_1)$, $(G_{34}, A_1^2)$ and $(G_{34}, A_2)$ 
are free with exponents given in \cite[Table C.17]{orlikterao:arrangements}.
Let $\CA^X$ be $(G_{34}, A_1)$ and let $H$ be any hyperplane in $\CA^X$.
By \cite[Table C.17]{orlikterao:arrangements},
$(\CA^X)''$ is either $(G_{34}, A_1^2)$ or $(G_{34}, A_2)$ 
and hence $\exp (\CA^X)'' \not\subseteq \exp \CA^X$. Thus by
Lemma \ref{lem:q}, $(\CA^X)'$ 
is not free and hence not inductively free. 
So $(G_{34}, A_1)$ is not inductively free.

As $(G_{34}, G(3, 3, 3)) \cong \CA(G_{26})$ 
(cf.\ \cite[App.~D]{orlikterao:arrangements}), 
it is inductively free, by Theorem \ref{thm:indfree1}(i).

Thanks to Theorem \ref{thm:indfree1}(i) and 
the classification of the 
restrictions $\CA^X$ from \cite[App.~C]{orlikterao:arrangements},
there are 10 cases that remain to be considered: 
$(G_{29}, A_1)$, $(G_{31}, A_1)$, $(G_{33}, A_1)$, $(G_{33}, A_1^2)$, $(G_{33}, A_2)$,
$(G_{34}, A_1^2)$, $(G_{34}, A_2)$, $(G_{34}, A_1^3)$, $(G_{34}, A_1A_2)$, 
and $(G_{34}, A_3)$.
We treated them computationally, see Remark \ref{rem:computations}.
It turns out that the 3-dimensional restrictions are still
inductively free, while the 4-dimensional ones are not.
In Tables \ref{table3} -- \ref{table8} we give the induction tables 
for the former instances; where we use $a, b$ and $c$ as
variable names for simplicity and where $i$ is a primitive 4-th root of 1 and 
$\zeta = e^{2\pi i/3}$.

Theorem \ref{thm:indfree}  now follows from Theorem \ref{thm:indfree1}(i)
and Lemmas \ref{lem:indfree-exeptional} and \ref{lem:notindfree-exeptional}
below.

\begin{lemma}
\label{lem:indfree-exeptional}
Each of the 3-dimensional restrictions 
$(G_{29}, A_1)$, $(G_{31}, A_1)$, $(G_{33}, A_1^2)$, $(G_{33}, A_2)$,
$(G_{34}, A_1^3)$, $(G_{34}, A_1A_2)$, and $(G_{34}, A_3)$ 
is inductively free.
\end{lemma}

\begin{proof}
We present the corresponding 
induction tables in Tables \ref{table3} -- \ref{table8} below.
\end{proof}

\begin{table}[h] 
\begin{tabular}[t]{lll}\hline
  $\exp\CA'$ & $\alpha_{H}$ & $\exp\CA''$\\ \hline\hline
  0, 0, 0 & $a + i  b - i  c$ & 0, 0\\
  0, 0, 1 & $a + i  b + i  c$ & 0, 1\\
  0, 1, 1 & $a - i  b + c$ & 1, 1\\
  1, 1, 1 & $a + b$ & 1, 1\\
  1, 1, 2 & $a + i  b + c$ & 1, 2\\
  1, 2, 2 & $c$ & 1, 2\\
  1, 2, 3 & $a - i  b - c$ & 1, 3\\
  1, 3, 3 & $a - i  b + i  c$ & 1, 3\\
  1, 3, 4 & $a - i  b - i  c$ & 1, 4\\
  1, 4, 4 & $a + i  b - c$ & 1, 4\\
  1, 4, 5 & $a$ & 1, 5\\
\hline 
\end{tabular}
  \qquad
\begin{tabular}[t]{lll}\hline
  $\exp \CA'$ & $\alpha_H$ & $\exp \CA''$\\ 
    \hline\hline 
  1, 5, 5 & $b$ & 1, 5\\
  1, 5, 6 & $a + b + i  c$ & 1, 6\\
  1, 6, 6 & $a + b - i  c$ & 1, 6\\
  1, 6, 7 & $a - b - i  c$ & 1, 7\\
  1, 7, 7 & $a - b + i  c$ & 1, 7\\
  1, 7, 8 & $a - c$ & 1, 8\\
  1, 8, 8 & $a - b$ & 1, 8\\
  1, 8, 9 & $b - c$ & 1, 9\\
  1, 9, 9 & $b + c$ & 1, 9\\
  1, 9, 10 & $a + c$ & 1, 9\\
  1, 9, 11 && \\ \hline
\end{tabular}
\medskip
\caption{Induction Table for $(G_{29}, A_1)$}
\label{table3}
\end{table}

\begin{table}[h] 
\begin{tabular}[t]{lll}\hline
  $\exp\CA'$ & $\alpha_{H}$ & $\exp\CA''$\\ \hline\hline
  0, 0, 0 & $b - i  c$ & 0, 0\\
  0, 0, 1 & $b - c$ & 0, 1\\
  0, 1, 1 & $a - i  c$ & 1, 1\\
  1, 1, 1 & $c$ & 1, 1\\
  1, 1, 2 & $b$ & 1, 1\\
  1, 1, 3 & $a$ & 1, 3\\
  1, 2, 3 & $a + i  b - i  c$ & 1, 3\\
  1, 3, 3 & $a + b - i  c$ & 1, 3\\
  1, 3, 4 & $a - b - i  c$ & 1, 4\\
  1, 4, 4 & $a - i  b - i  c$ & 1, 4\\
  1, 4, 5 & $a - i  b - c$ & 1, 5\\
  1, 5, 5 & $a - b + c$ & 1, 5\\
  1, 5, 6 & $a + i  b + c$ & 1, 6\\
  1, 6, 6 & $a - i  b + c$ & 1, 6\\
  1, 6, 7 & $a - b - c$ & 1, 7\\
  1, 7, 7 & $a + b - c$ & 1, 7\\
\hline 
\end{tabular}
  \qquad
\begin{tabular}[t]{lll}\hline
  $\exp \CA'$ & $\alpha_H$ & $\exp \CA''$\\ 
    \hline\hline 
  1, 7, 8 & $a + i  b - c$ & 1, 8\\
  1, 8, 8 & $a + b + c$ & 1, 8\\
  1, 8, 9 & $b + i  c$ & 1, 9\\
  1, 9, 9 & $b + c$ & 1, 9\\
  1, 9, 10 & $a - i  b + i  c$ & 1, 10\\
  1, 10, 10 & $a - b + i  c$ & 1, 10\\
  1, 10, 11 & $a + b + i  c$ & 1, 10\\
  1, 10, 12 & $a + i  b + i  c$ & 1, 10\\
  1, 10, 13 & $a - c$ & 1, 13\\
  1, 11, 13 & $a + i  c$ & 1, 13\\
  1, 12, 13 & $a + c$ & 1, 13\\
  1, 13, 13 & $a - i  b$ & 1, 13\\
  1, 13, 14 & $a - b$ & 1, 13\\
  1, 13, 15 & $a + i  b$ & 1, 13\\
  1, 13, 16& $a + b$ & 1, 13\\
  1, 13, 17 && \\ \hline
\end{tabular}
\medskip
\caption{Induction Table for $(G_{31}, A_1)$}
\label{table3a}
\end{table}

\begin{table}[h] 
\renewcommand{\arraystretch}{1.2}
\begin{tabular}[t]{lll}\hline
  $\exp\CA'$ & $\alpha_{H}$ & $\exp\CA''$\\ \hline\hline
  0, 0, 0 & $a$ & 0, 0\\
  0, 0, 1 & $a + 2b + c$ & 0, 1\\
  0, 1, 1 & $a - \zeta b - \frac{1}{2}\zeta^2c$ & 1, 1\\
  1, 1, 1 & $a + 2\zeta^2b + \zeta c$ & 1, 1\\
  1, 1, 2 & $a + 2\zeta b + \zeta^2c$ & 1, 2\\
  1, 2, 2 & $a - \zeta^2c$ & 1, 2\\
  1, 2, 3 & $a - \zeta c$ & 1, 3\\
  1, 3, 3 & $a - \zeta^2b - \frac{1}{2}\zeta c$ & 1, 3\\
  1, 3, 4 & $a + 2b - 2c$ & 1, 4\\
\hline 
\end{tabular}
  \qquad
\begin{tabular}[t]{lll}\hline
  $\exp \CA'$ & $\alpha_H$ & $\exp \CA''$\\ 
    \hline\hline 
  1, 4, 4 & $a - c$ & 1, 4\\
  1, 4, 5 & $a + 2\zeta b - 2\zeta^2c$ & 1, 5\\
  1, 5, 5 & $a - b - \frac{1}{2}c$ & 1, 5\\
  1, 5, 6 & $a + 2\zeta^2b - 2\zeta c$ & 1, 6\\
  1, 6, 6 & $c$ & 1, 6\\
  1, 6, 7 & $a - \zeta^2b + \zeta c$ & 1, 7\\
  1, 7, 7 & $a - b + c$ & 1, 7\\
  1, 7, 8 & $a - \zeta b + \zeta^2c$ & 1, 7\\
  1, 7, 9 && \\ \hline
\end{tabular}
\medskip
\caption{Induction Table for $(G_{33}, A_1^2)$}
\label{table4}
\end{table}

\begin{table}[h] 
\renewcommand{\arraystretch}{1.2}
\begin{tabular}[t]{lll}\hline
  $\exp\CA'$ & $\alpha_{H}$ & $\exp\CA''$\\ \hline\hline
  0, 0, 0 & $b - c$ & 0, 0\\
  0, 0, 1 & $b - \zeta c$ & 0, 1\\
  0, 1, 1 & $a - \frac{1}{2}b - \frac{3}{2}c$ & 1, 1\\
  1, 1, 1 & $b - \zeta^2c$ & 1, 1\\
  1, 1, 2 & $c$ & 1, 1\\
  1, 1, 3 & $a - \frac{1}{2} \zeta^2 b + \frac{1}{2} (\zeta^2 - 1)c$ & 1, 3  \\
  1, 2, 3 & $a - \frac{1}{2}\zeta b + \frac{1}{2}(\zeta - 1)c$ & 1, 3\\
\hline 
\end{tabular}
  \qquad
\begin{tabular}[t]{lll}\hline
  $\exp \CA'$ & $\alpha_H$ & $\exp \CA''$\\ 
    \hline\hline 
  1, 3, 3 & $a - \frac{1}{2}b$ & 1, 3\\
  1, 3, 4 & $a - \frac{1}{2}\zeta b - (\zeta + \frac{1}{2})c$ & 1, 4\\
  1, 4, 4 & $a - \frac{1}{2}\zeta^2b + (\zeta + \frac{1}{2})c$ & 1, 4\\
  1, 4, 5 & $a - \frac{1}{2}\zeta^2b - \frac{1}{2}(\zeta - 1)c$ & 1, 5\\
  1, 5, 5 & $a - \frac{1}{2}\zeta b - \frac{1}{2}(\zeta^2 - 1)c$ & 1, 5\\
  1, 5, 6 & $a - \frac{1}{2}b - \frac{3}{2}\zeta c$ & 1, 6\\
  1, 6, 6 & $a - \frac{1}{2}b - \frac{3}{2}\zeta^2c$ & 1, 6\\
  1, 6, 7 & & \\ \hline
\end{tabular}
\medskip
\caption{Induction Table for $(G_{33}, A_2)$}
\label{table5}
\end{table}

\begin{table}[h] 
\renewcommand{\arraystretch}{1.2}
\begin{tabular}[t]{lll}\hline
  $\exp\CA'$ & $\alpha_{H}$ & $\exp\CA''$\\ \hline\hline
  0, 0, 0 & $a - \zeta b$ & 0, 0\\
  0, 0, 1 & $b - \zeta^2c$ & 0, 1\\
  0, 1, 1 & $a - c$ & 0, 1\\
  0, 1, 2 & $a + (\zeta^2 - 1)b + 2\zeta^2c$ & 0, 1\\
  0, 1, 3 & $b - c$ & 1, 3\\
  1, 1, 3 & $a + (\zeta^2 - 1)b + 2c$ & 1, 3\\
  1, 2, 3 & $a + \frac{1}{3}(\zeta^2 - 1)b + \frac{2}{3}(\zeta^2 - 1)c$ & 1, 2\\
  1, 2, 4 & $a + \frac{1}{3}(\zeta^2 - 1)b - \frac{2}{3}(\zeta - 1)c$ & 1, 4\\
  1, 3, 4 & $a - \frac{1}{3}(\zeta^2 - 1)b + \frac{2}{3}(\zeta - 1)c$ & 1, 3\\
  1, 3, 5 & $a - \zeta c$ & 1, 5\\
  1, 4, 5 & $a - (\zeta^2 - 1)b + 2\zeta c$ & 1, 4\\
  1, 4, 6 & $a + \zeta b + \zeta^2c$ & 1, 6\\
  1, 5, 6 & $a - (\zeta^2 - 1)b + 2\zeta^2c$ & 1, 6\\
  1, 6, 6 & $a - \frac{1}{3}(\zeta^2 - 1)b + \frac{2}{3}(\zeta^2 - \zeta)c$ & 1, 6\\
  1, 6, 7 & $a + \zeta b$ & 1, 7\\
  1, 7, 7 & $a + \zeta b - 2\zeta^2c$ & 1, 7\\
  1, 7, 8 & $a + \zeta b + 2\zeta^2c$ & 1, 7\\
\hline 
\end{tabular}
  \qquad
\begin{tabular}[t]{lll}\hline
  $\exp \CA'$ & $\alpha_H$ & $\exp \CA''$\\ 
    \hline\hline 
  1, 7, 9 & $a + \zeta b - 2\zeta c$ & 1, 7\\
  1, 7, 10 & $a + \zeta b - 2c$ & 1, 7\\
  1, 7, 11 & $c$ & 1, 7\\
  1, 7, 12 & $a + \zeta b + 4\zeta^2c$ & 1, 7\\
  1, 7, 13 & $a - b$ & 1, 13\\
  1, 8, 13 & $a + 2\zeta^2c$ & 1, 13\\
  1, 9, 13 & $b - \zeta c$ & 1, 13\\
  1, 10, 13 & $a - \zeta b + 2(\zeta - 1)c$ & 1, 13\\
  1, 11, 13 & $b + 2\zeta c$ & 1, 13\\
  1, 12, 13 & $a - \zeta^2c$ & 1, 13\\
  1, 13, 13 & $a + \frac{1}{3}\zeta b + \frac{2}{3}\zeta^2c$ & 1, 13\\
  1, 13, 14 & $a + 3\zeta b + 2\zeta^2c$ & 1, 13\\
  1, 13, 15 & $a - \zeta b - 2(\zeta - 1)c$ & 1, 13\\
  1, 13, 16 & $a - \zeta^2b + (\zeta^2 - 1)c$ & 1, 13\\
  1, 13, 17 & $a - \zeta^2b$ & 1, 13\\
  1, 13, 18 & $a - b + (\zeta^2 - \zeta)c$ & 1, 13\\
  1, 13, 19 && \\ \hline
\end{tabular}
\medskip
\caption{Induction Table for $(G_{34}, A_1^3)$}
\label{table6}
\end{table}

\begin{table}[h] 
\renewcommand{\arraystretch}{1.2}
\begin{tabular}[t]{lll}\hline
  $\exp\CA'$ & $\alpha_{H}$ & $\exp\CA''$\\ \hline\hline
  0, 0, 0 & $c$ & 0, 0\\
  0, 0, 1 & $a + \zeta b + \zeta c$ & 0, 1\\
  0, 1, 1 & $a + \zeta b + (2\zeta - 1)c$ & 0, 1\\
  0, 1, 2 & $a - c$ & 1, 2\\
  1, 1, 2 & $a + 2\zeta c$ & 1, 2\\
  1, 2, 2 & $a - \zeta^2c$ & 1, 2\\
  1, 2, 3 & $a + \zeta b - (\zeta^2 - 2\zeta)c$ & 1, 3\\
  1, 3, 3 & $a - \zeta^2b - (\zeta^2 - \zeta)c$ & 1, 3\\
  1, 3, 4 & $b + 2c$ & 1, 4\\
  1, 4, 4 & $b - \zeta^2c$ & 1, 4\\
  1, 4, 5 & $a + \zeta b - 2c$ & 1, 5\\
  1, 5, 5 & $b - \zeta c$ & 1, 5\\
  1, 5, 6 & $b - (2\zeta^2 - 1)c$ & 1, 6\\
  1, 6, 6 & $a - b +(\zeta-1)c$ & 1, 6\\
  1, 6, 7 & $a - \zeta^2b + 3\zeta c$ & 1, 7\\
\hline 
\end{tabular}
  \qquad
\begin{tabular}[t]{lll}\hline
  $\exp \CA'$ & $\alpha_H$ & $\exp \CA''$\\ 
    \hline\hline  
 1, 7, 7 & $b - c$ & 1, 7\\
  1, 7, 8 & $a + (\zeta - 2)c$ & 1, 8\\
  1, 8, 8 & $a - \zeta^2b$ & 1, 8\\
  1, 8, 9 & $b - (2\zeta - 1)c$ & 1, 9\\
  1, 9, 9 & $a + \zeta b + 4 \zeta c$ & 1, 9\\
  1, 9, 10 & $a + \zeta b - 2\zeta^2c$ & 1, 10\\
  1, 10, 10 & $a + (3\zeta + 2)c$ & 1, 10\\
  1, 10, 11 & $a - \zeta c$ & 1, 10\\
  1, 10, 12 & $a + \frac{1}{2}\zeta b + \frac{3}{2}\zeta c$ & 1, 12\\
  1, 11, 12 & $a - b + 3\zeta c$ & 1, 12\\
  1, 12, 12 & $a + 2\zeta b + 3\zeta c$ & 1, 12\\
  1, 12, 13 & $a - \zeta b$ & 1, 13\\
  1, 13, 13 & $a - b$ & 1, 13\\
  1, 13, 14 & $a - b - 3c$ & 1, 13\\
  1, 13, 15 & $a - \zeta^2b - 3\zeta^2c$ & 1, 13\\
  1, 13, 16 && \\ \hline
\end{tabular}
\medskip
\caption{Induction Table for $(G_{34}, A_1A_2)$}
\label{table7}
\end{table}

\begin{table}[h] 
\renewcommand{\arraystretch}{1.2}
\begin{tabular}[t]{lll}\hline
  $\exp\CA'$ & $\alpha_{H}$ & $\exp\CA''$\\ \hline\hline
  0, 0, 0 & $a + \zeta b + \zeta^2 c$ & 0, 0\\
  0, 0, 1 & $a + b + c$ & 0, 1\\
  0, 1, 1 & $a - \zeta c$ & 0, 1\\
  0, 1, 2 & $b - \zeta^2 c$ & 0, 1\\
  0, 1, 3 & $a + \zeta^2 b - 2\zeta c$ & 0, 1\\
  0, 1, 4 & $a + \zeta^2 b + (2\zeta^2 - 1) c$ & 1, 4\\
  1, 1, 4 & $a - c$ & 1, 4\\
  1, 2, 4 & $a + \zeta^2 b + \zeta c$ & 1, 4\\
  1, 3, 4 & $a + b - 2 c$ & 1, 4\\
  1, 4, 4 & $c$ & 1, 4\\
  1, 4, 5 & $a + \zeta^2 b - (\zeta^2 - 2) c$ & 1, 5\\
  1, 5, 5 & $b - c$ & 1, 5\\
  1, 5, 6 & $a + b + (2\zeta^2 - \zeta) c$ & 1, 6\\
\hline 
\end{tabular}
  \qquad
\begin{tabular}[t]{lll}\hline
  $\exp \CA'$ & $\alpha_H$ & $\exp \CA''$\\ 
    \hline\hline  
 1, 6, 6 & $a + b - (\zeta^2 - 2\zeta) c$ & 1, 6\\
  1, 6, 7 & $a - \zeta b$ & 1, 7\\
  1, 7, 7 & $a + \zeta b - 2\zeta^2 c$ & 1, 7\\
  1, 7, 8 & $b - \zeta c$ & 1, 8\\
  1, 8, 8 & $a - \zeta^2 c$ & 1, 8\\
  1, 8, 9 & $a + \zeta b + (2\zeta - 1) c$ & 1, 9\\
  1, 9, 9 & $a + \zeta b - (\zeta - 2) c$ & 1, 9\\
  1, 9, 10 & $a + \zeta b + 4\zeta^2 c$ & 1, 9\\
  1, 9, 11 & $a + b + 4 c$ & 1, 11\\
  1, 10, 11 & $a - \zeta^2 b$ & 1, 11\\
  1, 11, 11 & $a + \zeta^2 b + 4\zeta c$ & 1, 11\\
  1, 11, 12 & $a - b$ & 1, 11\\
  1, 11, 13 && \\ \hline
\end{tabular}
\medskip
\caption{Induction Table for $(G_{34}, A_3)$}
\label{table8}
\end{table}

\begin{lemma}
\label{lem:notindfree-exeptional}
The 4-dimensional restrictions 
$(G_{33}, A_1)$,
$(G_{34}, A_1^2)$, and $(G_{34}, A_2)$
are not inductively free.
\end{lemma}

\begin{proof} 
The argument is similar to the proof in \cite[Lem.\ 3.5]{hogeroehrle:inductivelyfree},
where we showed that the reflection arrangement of the exceptional group 
of type $G_{31}$ is not inductively free.

Let $\CA$ be 
$(G_{33}, A_1)$,
$(G_{34}, A_1^2)$, or $(G_{34}, A_2)$,
respectively.
Then thanks to  \cite{hogeroehrle:free}, $\CA$ is known to be free
with exponents $\{1, 7, 9, 11\}$, $\{1, 13, 19, 23\}$,
or $\{1, 13, 16, 19\}$, respectively 
(cf.~\cite[Table C.14, Table C.17]{orlikterao:arrangements}).
Nevertheless,  one cannot successively remove all  hyperplanes 
from $\CA$ such that in each step the
resulting arrangement is free. 
In particular, $\CA$ can not be inductively free.
The Addition-Deletion Theorem \ref{thm:add-del}
and the fact that the sum of the exponents of a
free arrangement is the cardinality of the arrangement
(\cite[Thm.\ 4.23]{orlikterao:arrangements})
give a necessary condition for the freeness 
of these subarrangements in terms of 
the cardinalities
$\vert(\CA\setminus \{H_1,\ldots,H_l\})^{H_{l+1}}\vert$.
More precisely,  if $\{b_1,b_2,b_3,b_4\}$ is the set of 
exponents of $\CA\setminus \{H_1,\ldots,H_l\}$, then the 
condition is that 
\begin{equation}
\label{eq:cond}
\vert(\CA\setminus \{H_1,\ldots,H_l\})^{H_{l+1}}\vert = b_i + b_j + b_k
\quad \text{for } \{i,j,k\} \subset \{1, \ldots, 4\}.
\end{equation} 
If free, the new possible exponents 
of $\CA\setminus \{H_1,\ldots,H_{l+1}\}$ in these cases are
$\{b_1,b_2,b_3,b_4-1\}$, $\{b_1,b_2,b_3-1,b_4\}$, $\{b_1,b_2-1,b_3,b_4\}$, or $\{b_1-1,b_2,b_3,b_4\}$. 
Note that  
this arrangement might actually fail to be free. 

We determine all subarrangements $\CB = \{H_1,\ldots,H_n\}$ of $\CA$ 
for fixed  cardinality $n$ and count how many of them satisfy the 
necessary condition \eqref{eq:cond}. 
In Tables \ref{tablenif1} -- \ref{tablenif3} we give
the results of these computations in all three cases. 
More precisely, 
in the $n$-th row of each table
we give the number $N$ of all  
subarrangements $\CB$ of $\CA$ with $n$  hyperplanes  
which admit an ordering satisfying the 
necessary condition \eqref{eq:cond}.  
In the last column we list which 
possible exponents of $\CA \setminus \CB$ might occur in these cases. 
The computations were done with a 
breadth-first search, i.e.\ we first computed all cases in one row before proceeding 
to the next.
\end{proof}

\begin{table}[h] 
\begin{tabular}[t]{rrl}
$|\CB|$  & $N$    & $\exp (\CA \setminus \CB)$\\
\hline
\hline
1   &  12       & $\{1, 7, 9, 10\}$\\
2   &  48       & $\{1, 7, 9, 9\}$\\
3   &  48       & $\{1, 7, 8, 9\}$\\
4   & 144       & $\{1, 7, 8, 8\}$\\
5   &  72       & $\{1, 7, 7, 8\}$\\
6   &  12       & $\{1, 6, 7, 8\}$\\
7   &  48       & $\{1, 6, 7, 7\}$\\
8   &  72       & $\{1, 6, 6, 7\}$\\
9   &  48       & $\{1, 5, 6, 7\}$\\
10  &  12       & $\{1, 4, 6, 7\}$\\
11  &   0       &                \\
\hline
\end{tabular}
\medskip
\caption{The number of $\CB \subset \CA = (G_{33},A_1)$ 
satisfying \eqref{eq:cond}}
\label{tablenif1}
\end{table}

\begin{table}[h] 
\begin{tabular}[t]{rrl}
$|\CB|$  & $N$    &  $\exp(\CA \setminus \CB)$\\
\hline
\hline
1 &    12 &$\{1, 13, 19, 22\}$\\
2 &    66 &$\{1, 13, 19, 21\}$\\
3 &   204 &$\{1, 13, 19, 20\}$\\
4 &   351 &$\{1, 13, 19, 19\}$\\
5 &   288 &$\{1, 13, 18, 19\}$\\
6 &   432 &$\{1, 13, 17, 19\}$, $\{1, 13, 18, 18\}$\\
7 &   384 &$\{1, 13, 17, 18\}$, $\{1, 13, 16, 19\}$\\
8 &   351 &$\{1, 13, 17, 17\}$, $\{1, 13, 15, 19\}, \{1, 13, 16, 18\}$\\
9 &   172 &$\{1, 13, 16, 17\}$, $\{1, 13, 15, 18\}$\\
10&    98 &$\{1, 13, 16, 16\}$, $\{1, 13, 15, 17\}$\\
11&    28 &$\{1, 13, 15, 16\}$\\
12&     1 &$\{1, 13, 15, 15\}$\\
13&   0   &                \\
\hline
\end{tabular}
\medskip
\caption{The number of $\CB \subset \CA = (G_{34},A_1^2)$ 
satisfying \eqref{eq:cond}}
\label{tablenif2}
\end{table}

\begin{table}[h] 
\begin{tabular}[t]{rrl}
$|\CB|$  & $N$    & $\exp(\CA \setminus \CB)$\\
\hline
\hline
1&      18& $\{1, 13, 16, 18\}$\\
2&     126& $\{1, 13, 16, 17\}$\\
3&     402& $\{1, 13, 16, 16\}$\\
4&     612& $\{1, 13, 15, 16\}$\\
5&    1584& $\{1, 13, 15, 15\}$, $\{1, 13, 14, 16\}$\\
6&    2910& $\{1, 13, 13, 16\}$, $\{1, 13, 14, 15\}$\\
7&    6030& $\{1, 12, 13, 16\}$, $\{1, 13, 13, 15\}$, $\{1, 13, 14, 14\}$\\
8&    8865& $\{1, 13, 13, 14\}$, $\{1, 11, 13, 16\}$, $\{1, 12, 13, 15\}$\\
9&   12764& $\{1, 13, 13, 13\}$, $\{1, 12, 13, 14\}$, $\{1, 11, 13, 15\}$, $\{1, 10, 13, 16\}$\\
10&  11358& $\{1, 10, 13, 15\}$, $\{1, 12, 13, 13\}$, $\{1, 11, 13, 14\}$\\
11&   8982& $\{1, 12, 12, 13\}$, $\{1, 10, 13, 14\}$, $\{1, 11, 13, 13\}$\\
12&   8430& $\{1, 12, 12, 12\}$, $\{1, 11, 12, 13\}$, $\{1, 10, 13, 13\}$\\
13&   4491& $\{1, 11, 11, 13\}$, $\{1, 11, 12, 12\}$, $\{1, 10, 12, 13\}$\\
14&   2223& $\{1, 11, 11, 12\}$, $\{1, 10, 11, 13\}$, $\{1, 10, 12, 12\}$\\
15&   1068& $\{1, 10, 10, 13\}$, $\{1, 10, 11, 12\}$, $\{1, 11, 11, 11\}$\\
16&    261& $\{1, 10, 10, 12\}$, $\{1, 10, 11, 11\}$\\
17&    126& $\{1, 10, 10, 11\}$\\
18&     37& $\{1, 10, 10, 10\}$, $\{1, 9, 10, 11\}$ \\
19&      0&                \\
\hline
\end{tabular}
\medskip
\caption{The number of $\CB \subset \CA = (G_{34},A_2)$ 
satisfying \eqref{eq:cond}}
\label{tablenif3}
\end{table}

\begin{remark}
\label{rem:computations}
In order 
to establish the results of Lemmas \ref{lem:indfree-exeptional} and  \ref{lem:notindfree-exeptional}, 
we first use the functionality for complex reflection groups 
provided by the   \CHEVIE\ package in   \GAP\ 
(and some \GAP\ code by J.~Michel)
(see \cite{gap3} and \cite{chevie})
in order to obtain explicit 
linear functionals $\alpha$ defining the hyperplanes 
$H = \ker \alpha$ of the reflection arrangement
$\CA(W)$ and the relevant restrictions $\CA(W)^X$.
 
We then use the module theoretic functionality of
the   \Singular\ computer algebra system (cf.~\cite{singular})
to determine the induction tables in 
Tables \ref{table3} -- \ref{table8}, cf.\ 
Remark  \ref{rem:indtable}.

In addition we utilize 
the functionality of \Sage\ to 
compile the data in Tables \ref{tablenif1} -- \ref{tablenif3},
showing that  the $4$-dimensional restrictions 
of Lemma  \ref{lem:notindfree-exeptional} are not inductively 
free,  \cite{sage}.
\end{remark}

\subsection{Proof of Theorem \ref{thm:heredfree}}

The reverse implication of Theorem \ref{thm:heredfree} is obvious.
Thanks to \cite[Cor.~2.12]{hogeroehrle:inductivelyfree} and 
\eqref{eq:restrproduct},
the question of hereditary inductive freeness of a given restriction
of a product of arrangements reduces to the case
when $\CA$ is irreducible.

So suppose that $\CA = \CA(W)$ and $X \in L(\CA)$ are as in 
Theorem \ref{thm:indfree} so that $W$ is irreducible and
$\CA^X$ is inductively free.
If $\CA$ is as in part (i) of Theorem \ref{thm:indfree},
then, noting that if $Y \in L(\CA^X)$, then $(\CA^X)^Y = \CA^Y$, 
the result follows from 
Theorem \ref{thm:indfree1}(ii).
If $\CA$ is as in 
Theorem \ref{thm:indfree}(iii), then Lemmas \ref{lem:rank1and2} and
\ref{lem:hif} show that
$\CA^X$ is hereditarily inductively free.

Finally, suppose that $W = G(r,r,\ell)$ and $\CA(W)^X = \CA_p^k(r)$ 
for $p-2 \le k \le p$.
It follows from the fact that $p-2 \le k \le p$ and 
the restrictions in Table \ref{table2} that for any 
$H \in \CA_p^k(r)$, the restriction $(\CA_p^k(r))^H$ is of the form
$\CA_{p-1}^{k-1}(r)$, $\CA_{p-1}^k(r)$, $\CA_{p-1}^{k+1}(r)$,  or $\CA_{p-1}^{p-1}(r)$.  
It thus follows by induction and 
Theorem \ref{thm:indfree1}(ii) for $\CA(G(r,1,p)) \cong \CA_{p-1}^{p-1}(r)$
that $\CA_p^k(r)$ is hereditarily inductively free.

Our final result
follows readily from Theorems \ref{thm:indfree1}, 
\ref{thm:indfree}, \ref{thm:akl}(ii) and  
\cite[Prop.\ 2.14]{orliksolomon:unitaryreflectiongroups}
(cf.~\cite[Prop.~6.84]{orlikterao:arrangements}):

\begin{corollary}
\label{cor:recfree}
For $W$ a finite, irreducible, complex reflection group,  
let $\CA = \CA(W)$ be its reflection arrangement.
Then the following hold:
\begin{itemize}
\item[(i)] 
Suppose that 
$W$ is not isomorphic to  
$G_{24}, G_{27}, G_{29}, G_{31}, G_{33}$, or $G_{34}$.
Then $\CA^X$ is recursively free for any $X \in L(\CA)$.
In particular, $\CA$ is recursively free.
\item[(ii)] 
Suppose that  $W$ is one of 
$G_{24}, G_{27}, G_{29}, G_{31}, G_{33}$, or $G_{34}$
and that $X \in L(\CA) \setminus\{V\}$ with $\dim X \le 3$.
Then $\CA^X$ is recursively free.
\end{itemize}
\end{corollary}

In \cite[Rem.\ 3.7]{cuntzhoge},
Cuntz and the second author showed that  
$\CA(G_{27})$ fails to be recursively free.
This is the first known instance of a free but non-recursively free
arrangement.
It is not known whether the reflection arrangements
of $G_{24}, G_{29}, G_{31}, G_{33}$, and $G_{34}$ are recursively free.


\smallskip 
{\bf Acknowledgments}: 
We acknowledge 
support from the DFG-priority program 
SPP1489 ``Algorithmic and Experimental Methods in
Algebra, Geometry, and Number Theory''.


\bigskip

\bibliographystyle{amsalpha}

\newcommand{\etalchar}[1]{$^{#1}$}
\providecommand{\bysame}{\leavevmode\hbox to3em{\hrulefill}\thinspace}
\providecommand{\MR}{\relax\ifhmode\unskip\space\fi MR }
\providecommand{\MRhref}[2]{%
  \href{http://www.ams.org/mathscinet-getitem?mr=#1}{#2} }
\providecommand{\href}[2]{#2}


\end{document}